\documentclass[12pt,oneside]{amsart}

      \usepackage{amssymb}
\usepackage[margin=1in]{geometry}  
\usepackage{ulem}       
\usepackage{amsmath, amsthm, amssymb}
  \usepackage{verbatim} 
\usepackage{xspace}
\usepackage{mathrsfs}
\usepackage[pdftex,bookmarks,colorlinks,breaklinks]{hyperref}  
\hypersetup{linkcolor=blue,citecolor=red,filecolor=dullmagenta,urlcolor=darkblue} 
 \usepackage{graphicx}  
 \usepackage{color}	
 \usepackage{mathabx}
\newtheorem{theorem}{Theorem}[section]
\newtheorem{corollary}[theorem]{Corollary}
\newtheorem{lemma}[theorem]{Lemma}

\newtheorem{remark}[theorem]{Remark}

\newcommand{\scr}{\mathscr}   
      \makeatletter
      \def\@setcopyright{}
      \def\serieslogo@{}
      \makeatother
   
\begin{document}

%



   \author{Amin  Bahmanian}
   \address{Department of Mathematics and Statistics
 Auburn University, Auburn, AL USA   36849-5310}
   \curraddr{Department of Mathematics and Statistics, University of Ottawa, 585 King Edward, 
     Ottawa, ON Canada  K1N 6N5 }

   \email{mbahmani@uottawa.ca}



   \title[Connected Baranyai Theorem]{Connected Baranyai's Theorem} 


   \begin{abstract}
Let $K_n^h=(V,\binom{V}{h})$ be the complete $h$-uniform hypergraph on vertex set $V$ with $|V|=n$. Baranyai showed that  $K_n^h$ can be expressed as the union of edge-disjoint $r$-regular  factors if and only if $h$ divides $rn$ and $r$ divides $\binom{n-1}{h-1}$.  
Using a new proof technique, in this paper we prove that $\lambda K_n^h$ can be expressed as the union $\mathcal G_1\cup \ldots \cup\mathcal G_k$ of $k$ edge-disjoint  factors, where for $1\leq i\leq k$, $\mathcal G_i$ is $r_i$-regular, if and only if (i) $h$ divides $r_in$ for  $1\leq i\leq  k$, and (ii) $\sum_{i=1}^k r_i=\lambda \binom{n-1}{h-1}$. Moreover, for any $i$ ($1\leq i\leq k$) for which $r_i\geq 2$,  this new technique allows us to  guarantee that $\mathcal G_i$ is connected, 
 generalizing Baranyai's theorem, and answering a question by Katona.
   \end{abstract}



   \date{\today}


   \maketitle



\section {Introduction} 
A {\it hypergraph} $\mathcal G$ is a pair $(V,E)$ where $V$ is a finite set called the vertex set, $E$ is the edge multiset, where every edge is itself a multi-subset of $V$. This means that not only can an edge  occur multiple times in $E$, but also each vertex can have multiple occurrences within an edge.  
The total number of occurrences of a  vertex $v$ among all edges of $E$ is called the {\it degree}, $d_{\mathcal G}(v)$ of $v$ in $\mathcal G$. For a positive integer $h$, $\mathcal G$ is said to be $h$-{\it uniform} if $|e|=h$ for each $e\in E$.   
For  positive integers $r, r_1,\dots,r_k$, an $r$-factor in a hypergraph $\mathcal G$ is a spanning  $r$-regular sub-hypergraph, and an {\it $(r_1,\dots,r_k)$-factorization}  is a partition of the edge set of $\mathcal G$ into $F_1,\dots, F_k$ where $F_i$ is an $r_i$-factor for $1\leq i\leq k$, abbreviate $(r,\dots,r)$-factorization to $r$-factorization.  
The hypergraph $K_n^h:=(V,\binom{V}{h})$ with $|V|=n$ (by $\binom{V}{h}$ we mean the collection of all $h$-subsets of $V$) is called a  {\it complete} $h$-uniform hypergraph. Avoiding trivial cases, we assume that $n>h$. 
 Baranyai proved that:
\begin{theorem}\textup{(Baranyai \cite{Baran75})}\label{baranyaift}
If $a_1,\dots,a_s$ are positive integers such that $\sum_{i=1}^s a_i=\binom{n}{h}$, then the edges of $K_n^h=(V,E)$ can be partitioned into almost regular hypergraphs $(V,E_i)$ so that $|E_i|=a_i$ for $1\leq i\leq s$. 
\end{theorem}
In particular,  if $h\divides r_in$ and $\sum_{i=1}^k r_i=\lambda \binom{n-1}{h-1}$, then $K_n^h$  is $(r_1,\dots,r_k)$-factorizable. It is  natural to ask if we can obtain a connected factorization; that is, a factorization in which each factor   is a connected hypergraph.  
Let $m$ be the least common multiple of $h$ and $n$, and let $a=m/h$. Define the set of edges
$$\mathscr K=\{ \{1,\dots,h\}, \{h+1,\dots,2h\},\dots, \{(a-1)h+1,(a-1)h+2,\dots,ah\}\},$$ 
where the elements of the edges are considered mod $n$. The families obtained from $\mathscr K$ by permuting the elements of the underlying set $\{n\}$ are called {\it wreaths}.  If $h$ divides $n$, then a wreath is just a partition. Baranyai and Katona conjectured that the edge set of $K_n^h$ can be decomposed into disjoint wreaths \cite{Kat91}. In connection with this conjecture, Katona (private communication) suggested the problem of finding a connected factorization for $K_n^h$. In this paper, we solve this problem. 

If we replace every edge $e$ of $K_n^h$ by $\lambda$ copies of $e$, then we denote the new hypergraph by $\lambda K_n^h$.   In this paper, the main result is the  following theorem:
\begin{theorem}\label{connbaranyaift}
 $\lambda K_n^h$ is $(r_1,\dots,r_k)$-factorizable if and only if $h\divides r_in$  for $1\leq i\leq k$, and $\sum_{i=1}^k r_i=\lambda \binom{n-1}{h-1}$.
 Moreover, for $1\leq i\leq k$, if $r_i\geq 2$, then we can guarantee that the $r_i$-factor is connected.
 \end{theorem}
In particular if $\lambda=1$, and $h=r_1=\dots=r_k=2$, Theorem \ref{connbaranyaift} implies the classical result of Walecki \cite{L} that the edge set of $K_n$ can be partitioned into Hamiltonian cycles if and only if $n$ is odd. Here we list some other interesting special consequences of Theorem \ref{connbaranyaift}:

\begin{corollary}\label{connbaranyaiftcor1}
 $K_n^h$ is connected $2$-factorizable if and only if $\binom{n-1}{h-1}$ is even and $h\divides 2n$.
\end{corollary}
\begin{corollary}\label{connbaranyaiftcor2}
$K_n^h$ has a connected $\frac{h}{\gcd(n,h)}$-factorization.
\end{corollary}

We note that the idea behind the proof of Theorem \ref{connbaranyaift} is based on the amalgamation technique; for some graph amalgamation results, see \cite{BahRod1, BahRod2, H2, HR, MatJohns, Nash87}  and for hypergraph amalgamations, see \cite{BahHyp1, BahHyp2, BahRod3}. Preliminaries are given in Section \ref{prelim}, followed by the proof of Theorem \ref{connbaranyaift} in Section \ref{proofsmt}. 

We end this section with some notation we need to be able to describe hypergraphs that arise in this setting. 

Let $\mathcal G=(V,E)$ be a hypergraph with $\alpha\in V$, and let $U=\{u_1,\dots,u_z\}\subset V\backslash \{\alpha\}$. Recall that each edge is a multi-subset of $V$. We abbreviate an edge of the form  $\{{\underbrace{\alpha,\ldots,\alpha}_p},u_1,\dots,u_z\}$ to $\{\alpha^p,u_1,\dots,u_z\}$. An {\it $h$-loop} incident with $\alpha$ is an edge of the form
$\{\alpha^h\}$, and $m(\alpha^p,U)$ denotes the multiplicity of an edge of the form $\{\alpha^p\}\cup U$. A \textit{k-edge-coloring} of $\mathcal G$ is a mapping $f: E\rightarrow C$, where $C$ is a set of $k$ \textit{colors} (often we use $C=\{1,\ldots,k\}$), and the edges of one color form a \textit{color class}. The sub-hypergraph of $\mathcal G$ induced by the color class $i$ is denoted by $\mathcal G_i$,  abbreviate $d_{\mathcal G_i}(\alpha)$ to $d_i(\alpha)$ and  $m_{\mathcal G_i}(\alpha^p,U)$ to $m_{i}(\alpha^p,U)$. 

\section{Preliminaries}\label{prelim}
A hypergraph is said to be {\it non-trivial} if it has at least one edge. A vertex $\alpha$ in a connected hypergraph $\mathcal G$ is a {\it cut vertex} if there exist two non-trivial sub-hypergraphs $I, J$ of $\mathcal G$ such that $I\cup J=\mathcal G$, $V( I\cap  J)=\alpha$ and $E( I\cap  J)=\varnothing$.  A non-trivial connected sub-hypergraph $W$ of a connected hypergraph $\mathcal G$ is said to be an $\alpha$-wing of $\mathcal G$, if $\alpha$ is not a cut vertex of $W$ and no edge in $E(\mathcal G)\backslash E(W)$ is incident with a vertex in $V(W)\backslash \{\alpha\}$. 
 The set of all $\alpha$-wings of $\mathcal G$ is denoted by $\scr W_\alpha (\mathcal G)$. We remark that $\scr W_\alpha (\mathcal G)=\{\mathcal G\}$ if $\mathcal G$ is non-trivial and connected and $\alpha$ is not a cut vertex of $\mathcal G$. 
 Figure \ref{figure:allwings} illustrates an example of a hypergraph and the set of all its $\alpha$-wings. 
\begin{figure}[htbp]
\begin{center}
\scalebox{.50}
{ \includegraphics {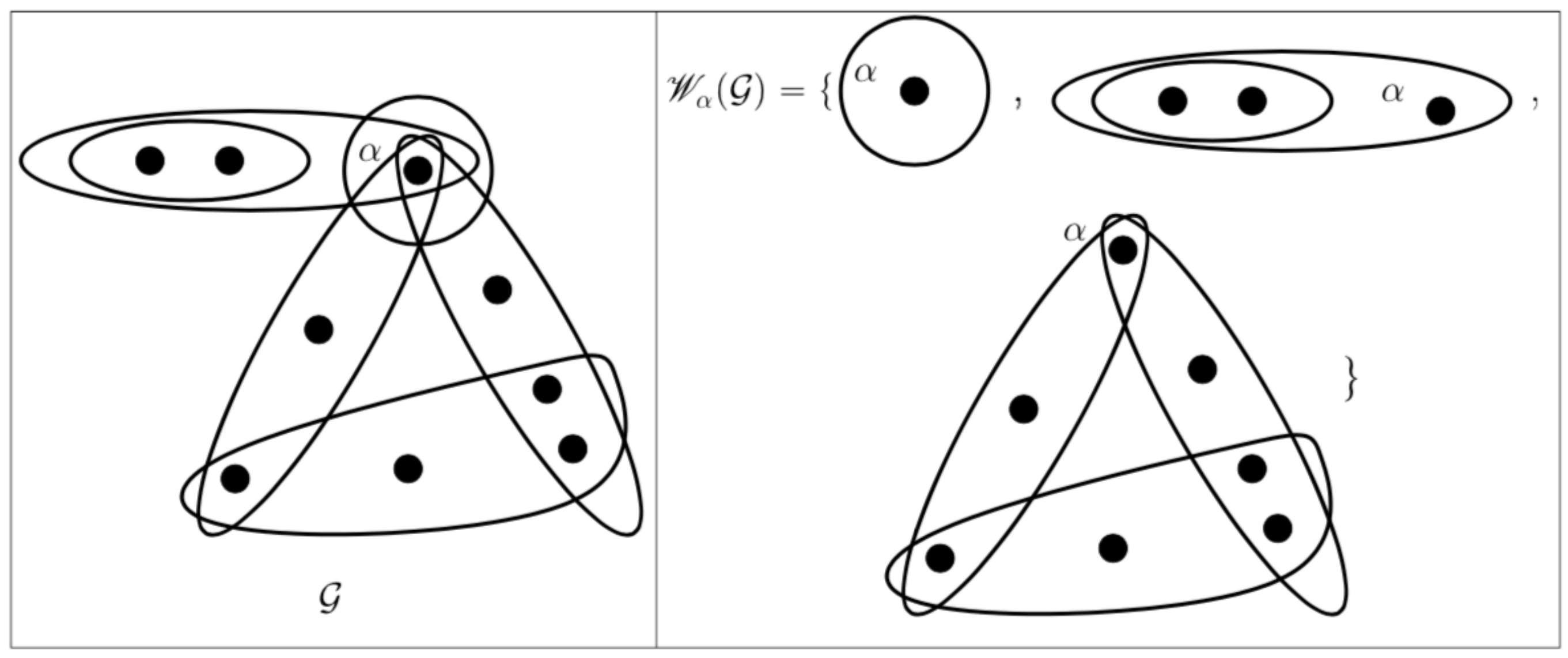} }
\caption{A hypergraph $\mathcal G$ and the set of all its $\alpha$-wings }
\label{figure:allwings}
\end{center}
\end{figure}
If the multiplicity of a vertex $\alpha$ in an edge $e$ is $p$,  we say that $\alpha$ is {\it incident} with $p$ distinct objects, say $h_1(\alpha,e),\dots,h_p(\alpha,e)$. We call these objects {\it hinges}, and we say that $e$ is {\it incident} with  $h_1(\alpha,e),\dots,h_p(\alpha,e)$. The set of all hinges  in $\mathcal G$ incident with $\alpha$ is denoted by $H_{\mathcal G}(\alpha)$; so $|H_{\mathcal G}(\alpha)|$ is in fact  the degree of $\alpha$. 

Intuitively speaking, an {\it $\alpha$-detachment} of $\mathcal G$ is a hypergraph obtained by splitting a vertex $\alpha$ into one or more vertices and sharing the incident hinges and edges  among the subvertices. That is, in an $\alpha$-detachment $\mathcal G'$ of $\mathcal G$ in which we split $\alpha$ into $\alpha$ and $\beta$,  an edge of the form $\{\alpha^p,u_1,\dots,u_z\}$ in $\mathcal G$ will be of the form $\{\alpha^{p-i},\beta^{i},u_1,\dots,u_z\}$ in $\mathcal G'$ for some $i$, $0\leq i\leq p$. Note that a hypergraph and its detachments have the same hinges. Whenever it is not ambiguous, we use $d'$, $m'$, etc. for degree, multiplicity and other hypergraph parameters in $\mathcal G'$.  Also, for an $\alpha$-wing $W$ in $\mathcal G$ and an $\alpha$-detachment $\mathcal G'$, let $W'$  denote the sub-hypergraph of $\mathcal G'$ whose hinges are the same as those in $W$.  Figure \ref{figure:allwingsdetach} illustrates a detachment  $\mathcal G'$ of the hypergraph $\mathcal G$ in Figure \ref{figure:allwings}  and the set of all its $\alpha$-wings. 
\begin{figure}[htbp]
\begin{center}
\scalebox{.50}
{ \includegraphics {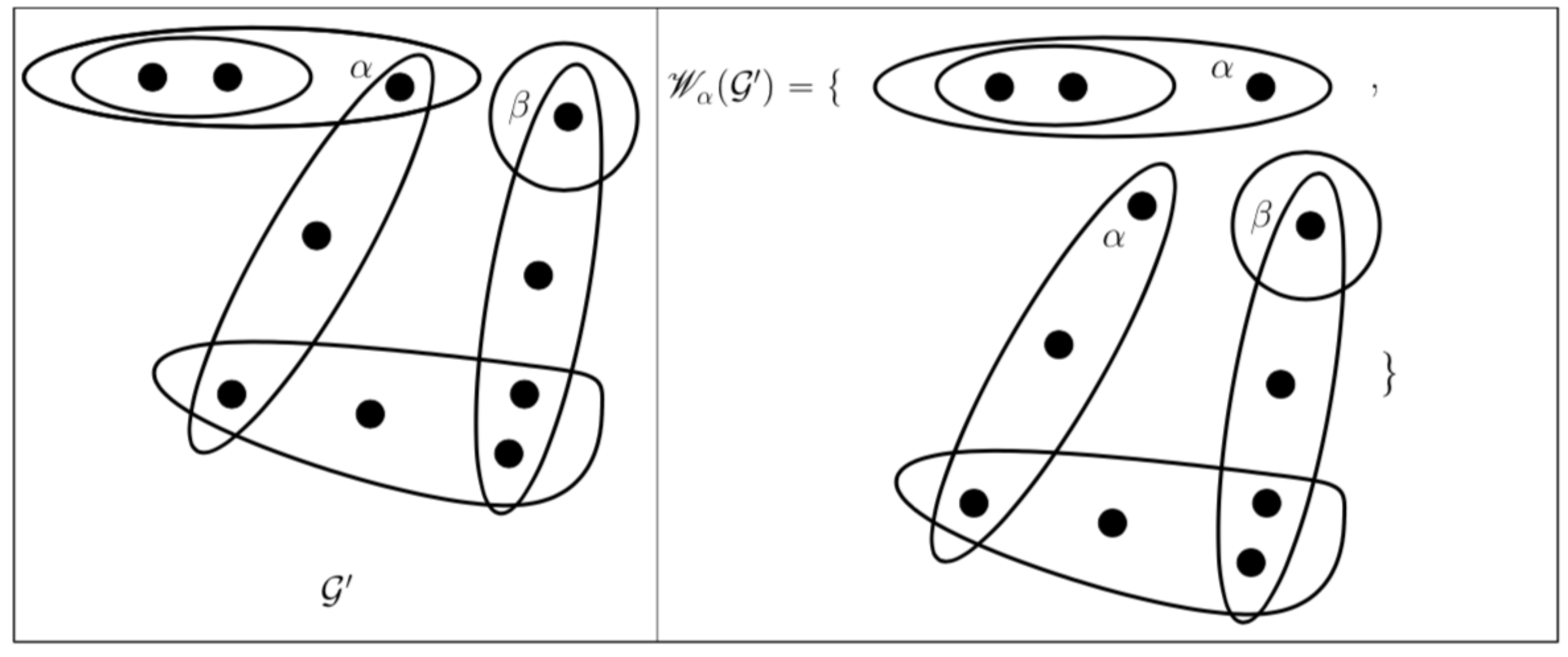} }
\caption{A detachment $\mathcal G'$ of $\mathcal G$ in Figure \ref{figure:allwings} and the set of all its $\alpha$-wings }
\label{figure:allwingsdetach}
\end{center}
\end{figure}
We shall present three lemmas, all of which follow immediately from definitions. 
\begin{lemma} \label{simpleconlem}
Let  $\mathcal G$ be a connected hypergraph. Let $\mathcal G'$ be an $\alpha$-detachment of $\mathcal G$ obtained by splitting a vertex $\alpha$ into two vertices $\alpha$ and $\beta$. Then $\mathcal G'$ is connected if and only if for some $\alpha$-wing  $W\in \scr W_{\alpha}(\mathcal G)$ with $d_W(\alpha)\geq 2$, 
$$1\leq |H_W(\alpha) \cap H_{\mathcal G'}(\beta)| < d_W(\alpha).$$
\end{lemma}
Informally speaking, Lemma \ref{simpleconlem} says that for some $\alpha$-wing $W$ with $d_W(\alpha)\geq 2$, at least one but not all the hinges incident with $\alpha$ in $W$ must be incident with $\beta$ in $\mathcal G'$. 

A family $\scr A$ of sets is \textit{laminar} if, for every pair $A, B$ of sets belonging to $\scr A$: $A\subset B$, or $B\subset A$, or $A\cap B=\varnothing$. 

Let us fix a vertex $\alpha$ of a $k$-edge-colored hypergraph $\mathcal G=(V,E)$. For $1\leq i\leq k$, let $H_i$ be the set of hinges each of which is incident with both $\alpha$ and an edge of color $i$ (so $d_i(\alpha)=|H_i|$). For any edge $e\in E$, let $H_e$ be the collection of hinges incident with both $\alpha$ and  $e$. Clearly, if $e$ is of color $i$, then $H_e\subset H_i$. For an $\alpha$-wing $W$, let $H_W=H_W(\alpha)$.   
For $1\leq i\leq k$, let 
$$
H^i=\bigcup_{W\in \scr W_\alpha(\mathcal G_i), d_W(\alpha)\geq 2} H_W.
$$

\begin{lemma}\label{lamAlem}
Let \begin{eqnarray*}
\scr A & = &   \{H_{1},\ldots,H_{k}\} \cup   \{H^1,\ldots,H^k\} \nonumber\\
& \cup & \{ H_{W} : W  \in \scr W_{ \alpha}(\mathcal G_i), 1\leq i \leq k \}   \cup  \{ H_e : e  \in E\}\nonumber. 
\end{eqnarray*}
Then $\scr A$ is a laminar family of subsets of $H(\alpha)$.
\end{lemma}
For each $p\geq 1$, and each $U \subset V\backslash\{\alpha\}$, let $H^U_p$ be the set of hinges each of which is incident with both $\alpha$ and an edge of the form $\{\alpha^p\}\cup U$ in $\mathcal G$ (so $|H_p^U|=pm(\alpha^p,U)$).
\begin{lemma}\label{lamBlem}
Let 
\begin{eqnarray*} 
\scr B&=&\{H_p^U: p\geq 1, U \subset V\backslash\{\alpha\}\}.
\end{eqnarray*}
Then $\scr B$ is a laminar family of disjoint subsets of $H(\alpha)$.
\end{lemma}

If $x,y$ are real numbers, $x\approx y$ means $\lfloor y \rfloor \leq x\leq \lceil y \rceil$. We need the following powerful lemma:
\begin{lemma}\textup{(Nash-Williams \cite[Lemma 2]{Nash87})}\label{laminarlem}
If $\scr A, \scr B$ are two laminar families of subsets of a finite set $S$, and $n$ is a positive integer, then there exist a subset $A$ of $S$ such that 
\begin{eqnarray*} 
 |A\cap P|\approx |P|/n \mbox { for every } P\in \scr A \cup \scr B. 
\end{eqnarray*} 
\end{lemma}

\section{Proof of the Main Theorem}\label{proofsmt}
To prove Theorem \ref{connbaranyaift}, first we look at the obvious necessary conditions:
\begin{lemma}\label{connbaranyaiftnecc}
If  $\lambda K_n^h$ is connected $(r_1,\dots,r_k)$-factorizable, then 
\begin{itemize}
\item[(i)]  $r_i\geq 2$  for $1\leq i\leq k$, 
\item [(ii)] $h\divides r_in$  for $1\leq i\leq k$, and 
\item [(iii)] $\sum_{i=1}^k r_i=\lambda \binom{n-1}{h-1}$.
\end{itemize}
\end{lemma}
\begin{proof} Suppose that $\lambda K_n^h$ is connected $(r_1,\dots,r_k)$-factorizable. The necessity of (i) is sufficiently obvious. Since each edge contributes $h$ to the the sum of the degrees of the vertices in an $r_i$-factor for $1\leq i\leq k$,  we must have (ii). Since each $r_i$-factor is an $r_i$-regular spanning sub-hypergraph for $1\leq i\leq k$, and $\lambda K_n^h$ is  $\lambda \binom{n-1}{h-1}$-regular, we must have (iii).  
\end{proof}

In order to get an inductive proof of Theorem \ref{connbaranyaift} to work, we actually prove the following  seemingly stronger result:
\begin{theorem}\label{seemstmainres}
Let $n, h, \lambda, k, r_1,\dots,r_k$ be positive integers with $n>h$ satisfying \textup{ (i)--(iii)}. For any integer $1\leq \ell \leq n$, there exists an $\ell$-vertex $k$-edge-colored $h$-uniform hypergraph $\mathcal G$ with vertex set $V$ ($\alpha\in V$) 
such that 
\begin{equation} \label{degcolcond}
d_i(u) = \left \{ \begin{array}{ll}
r_i(n-\ell+1) & \mbox { if } u = \alpha  \\
r_i & \mbox { if } u\neq \alpha \end{array} \right. \mbox { for } u\in V,  1\leq i\leq k, 
\end{equation}

\begin{equation} \label{multcond}
 m(\alpha^p,U)=\lambda \binom{n-\ell+1}{p} \mbox { for } p\geq 0, U\subset V\backslash\{\alpha\}  \mbox { with }  |U|=h-p, \mbox { and }  
\end{equation} 
\begin{equation} \label{conncondit}
\mathcal G_i \mbox { is connected if } r_i\geq 2, \mbox{ for } 1\leq i\leq k.
\end{equation} 
\end{theorem}
\begin{remark} \textup{
Theorem \ref{connbaranyaift} follows from Theorem \ref{seemstmainres} in the case where $\ell=n$ as the following argument shows.  If $\ell=n$, then conditions (\ref{degcolcond})--(\ref{conncondit}) imply that 
we have an $n$-vertex $k$-edge-colored hypergraph $\mathcal G$ in which the $i^{th}$ color class is $r_i$-regular by (\ref{degcolcond}), and connected by (\ref{conncondit}). Moreover,  (\ref{multcond}) implies that for $U\subset V\backslash\{\alpha\}$, (i) $m(U)=\lambda \binom{1}{0}=\lambda$ if  $|U|=h$ (when $p=0$), (ii) $m(\alpha,U)=\lambda \binom{1}{1}=\lambda$ if  $|U|=h-1$ (when $p=1$), and (iii) $m(\alpha^p,U)=\lambda \binom{1}{p}=0$ for $p\geq 2$, and $|U|=h-p$. Therefore $\mathcal G\cong \lambda K_n^h$. 
}\end{remark}
\begin{proof} 
The proof is by induction on $\ell$. At each step we will assume  not only  that $\mathcal G$ is  an $\ell$-vertex $k$-edge-colored hypergraph with vertex set $V$ ($\alpha\in V$) satisfying conditions (\ref{degcolcond})--(\ref{conncondit}), but that $\mathcal G$ also satisfies the two additional properties 
\begin{equation} \label{simplicity}
|H_e|\leq n-\ell+1 \mbox{ for each edge } e  \mbox { of } \mathcal G, \mbox{ and }
\end{equation}  
\begin{equation}\label{wingeq}
\mbox{ for } 1\leq i\leq k, \mbox{ if } r_i\geq 2 \mbox { and if } \ell\leq n-1,\mbox{ then } \delta_i =r_i(n-\ell+1) 
\end{equation}
where $\delta_i=|H^i|$ for $1\leq i\leq k$.

First consider the base case when $\ell=1$. Let $\mathcal F$ be a hypergraph with a single vertex $\alpha$ incident with $\lambda \binom{n}{h}$ $h$-loops; i.e. $m(\alpha^h)=\lambda\binom{n}{h}$. Color the edges of $\mathcal F$ such that $m_i(\alpha^h)=r_in/h$ for $1\leq i\leq k$. This is possible since by (ii) $h\divides r_in$, and by (iii) $\sum_{i=1}^k m_i(\alpha^h)=\sum_{i=1}^k r_in/h=n/h\sum_{i=1}^k r_i=\lambda n \binom{n-1}{h-1}/h=\lambda\binom{n}{h}=m(\alpha^h)$. Also, note  that for $\ell=1$, the hypergraph $\mathcal F$ trivially  satisfies (\ref{simplicity}), and since each $h$-loop is an $\alpha$-wing, $\mathcal F$ also satisfies (\ref{wingeq}).
Therefore, $\mathcal F$ shows that  conditions (\ref{degcolcond})--(\ref{wingeq})  holds for $\ell=1$. 

Now suppose that  $1\leq \ell < n$, and that $\mathcal G$  satisfies (\ref{degcolcond})--(\ref{wingeq}). The proof is completed by showing that $\mathcal G$ has an $(\ell+1)$-vertex $\alpha$-detachment  $\mathcal G'$ with vertex set $V'=V\cup \{\beta\}$ satisfying
  
\begin{equation} \label{simplicity'}
|H'_e|\leq n-\ell \mbox{ for each edge } e  \mbox { of } \mathcal G', 
\end{equation}
 \begin{equation} \label{degcolcond'}
d'_i(u) = \left \{ \begin{array}{ll}
r_i(n-\ell) & \mbox { if } u = \alpha  \\
r_i & \mbox { if } u\neq \alpha \end{array} \right. \mbox { for } u\in V', 1\leq i\leq k, 
 \end{equation}
\begin{equation} \label{multcond'}
 m'(\alpha^p,U)=\lambda \binom{n-\ell}{p} \mbox { for } p\geq 0, U\subset V'\backslash\{\alpha\} \mbox { with } |U|=h-p, 
\end{equation} 
\begin{equation} \label{conncondit'}
\mathcal G'(i) \mbox { is connected if } r_i\geq 2, \mbox{ for } 1\leq i\leq k, \mbox { and }  
\end{equation} 
for $1\leq i\leq k$, if $r_i\geq 2$ and if $\ell<n-1$,    then
\begin{equation} \label{wingdeg'}
\delta'_i =r_i(n-\ell).
\end{equation}

Let $\scr A$ and $\scr B$ be the  laminar families  in Lemmas \ref{lamAlem}, and \ref{lamBlem}. By Lemma \ref{laminarlem}, there exists a subset $A$ of $H(\alpha)$ such that 
\begin{equation}
|A\cap P|\approx |P|/(n-\ell+1) \mbox{ for every   } P\in \scr A \cup \scr B.
\end{equation}
Let $\mathcal G'$ be the hypergraph obtained from $\mathcal G$ by splitting $\alpha$ into two vertices $\alpha$ and  $\beta$ in such a way that hinges which were incident with $\alpha$ in $ \mathcal G$ become incident in $\mathcal G'$ with $\alpha$ or $\beta$ according as they do not or do belong to $A$, respectively. More precisely,
 \begin{equation}\label{hinge1'} 
 H'(\beta)=A, \quad H'(\alpha)=H( \alpha)\backslash A. 
 \end{equation}

Since $H_i\in \scr A$ for $1\leq i\leq k$, we have  
\begin{eqnarray*}
d'_i(\beta)& = & |A\cap H_i|\\
& \approx & |H_i|/(n-\ell+1)=d_i(\alpha)/(n-\ell+1) \\
& = &r_i(n-\ell+1)/(n-\ell+1)= r_i,\\
d'_i(\alpha) & = &   d_i(\alpha) - d'_i(\beta) \nonumber \\
& =&  r_i(n-\ell+1)-r_i=r_i(n-\ell),
\end{eqnarray*}
and for $u\notin\{\alpha,\beta\}$, $d'_i(u)=d_i(u)=r_i$. Therefore $\mathcal G'$ satisfies (\ref{degcolcond'}). 

Let $e$ be an edge in $\mathcal G$ incident with $\alpha$. Then $H_e\in \scr A$, and so
$$|A\cap H_e|\approx |H_e|/(n-\ell+1)\leq 1,$$ 
observing that the last inequality implies from (\ref{simplicity}). This means that either $A\cap H_e=\varnothing$  or $|A\cap H_e|=1$. Therefore $m'(\beta^q,U)=0$ for $q\geq 2$ and $U\subset V'$. Also, note that if $|H_e|=n-\ell+1$, then $|A\cap H_e|=1$ and thus $|H'_e|=n-\ell$, and if $|H_e|<n-\ell+1$, then $|H'_e|\leq |H_e|\leq n-\ell$, both cases together proving (\ref{simplicity'}). 

Since for $p\geq 1$, and $U\subset V\backslash \{\alpha\}$,  $H_p^U\in \scr B$, we have
\begin{eqnarray*}
m'(\alpha^{p-1}, \beta, U)& = & |A\cap H_p^U|\\
& \approx & |H_p^U|/(n-\ell+1)=pm(\alpha^p, U)/(n-\ell+1) \\
& = &\lambda p\binom{n-\ell+1}{p}/(n-\ell+1)= \lambda\binom{n-\ell}{p-1},\\
m'(\alpha^p, U) & = &   m(\alpha^p,U) -  m '(\alpha^{p-1}, \beta, U) \nonumber \\
& =& \lambda \binom{n-\ell+1}{p}-\lambda \binom{n-\ell}{p-1}=\lambda \binom{n-\ell}{p}.
\end{eqnarray*}
Therefore $\mathcal G'$ satisfies (\ref{multcond'}).

Let us fix an $i$, $1\leq i\leq k$ such that $r_i\geq 2$. Let $W$ be an $\alpha$-wing of $\mathcal G_i$ with $d_W(\alpha)\geq 2$. Then $H_W\in \scr A$, and so
\begin{equation}\label{conreq0}
|A\cap H_W|\approx |H_W|/(n-\ell+1)=d_W(\alpha)/(n-\ell+1),
 \end{equation} 
which implies that (noting that $n-\ell+1\geq 2$)
\begin{equation}\label{conreq1}
|A\cap H_W|<|H_W|.
\end{equation} 
Moreover,
\begin{equation}\label{conreq2}
|A\cap H^i|\approx |H^i|/(n-\ell+1)=\delta_i/(n-\ell+1)=r_i\geq 2,
\end{equation}
and therefore there exists an $\alpha$-wing $W$ in $\mathcal G_i$  with $d_W(\alpha)\geq 2$, such that $A\cap H_W\neq \varnothing$. Therefore by Lemma \ref{simpleconlem}, 
$\mathcal G'_i$ is connected. 

Now, suppose that $\ell\leq n-2$, or equivalently that $n-\ell+1\geq 3$. Since $\delta_i=d_i$ by (\ref{degcolcond}) and (\ref{wingeq}), we have that for every $W\in \mathscr W_\alpha(\mathcal G_i)$, $d_W(\alpha)\geq 2$. So there is no $\alpha$-wing $W$ in $\mathcal G_i$ with $d_W(\alpha)=1$. Let us fix an $\alpha$-wing $W$ in $\mathcal G_i$. There are two cases to consider:
\begin{itemize}
\item Case 1: If $|H_W|\geq 3$, then since $|A\cap H_W|\approx |H_W|/(n-\ell+1)\leq |H_W|/3$, we have that $d'_{W'}(\alpha)\geq 2$, and 
thus $\delta'_i=d'_i(\alpha)=r_i(n-\ell)$. Note that $W'$ is a sub-hypergraph of some $\alpha$-wing $S$ in $\mathcal G'$ with $d'_S(\alpha)\geq 2$. 

\item Case 2: If $|H_W|=2$, then $|A\cap H_W|\approx |H_W|/(n-\ell+1)=2/(n-\ell+1) \leq  2/3$. So $|A\cap H_W|\in \{0,1\}$. If $A\cap H_W=\varnothing$, we are done. So let us assume that $|A\cap H_W|=1$. Recall from (\ref{conreq2}) that $|A\cap H^i|\geq 2$. Therefore, there is another $\alpha$-wing $T$ in $\mathcal G_i$ with $|H_T|\geq 2$ such that $1\leq |A\cap H_T|<|H_T|$. Therefore, there exists an $\alpha$-wing $S$ in $\mathcal G'$ with $W'\cup T'\subset S$, and $d'_S(\alpha)\geq 2$.
Thus, in this case also we have $\delta'_i=\delta_i-r_i=r_i(n-\ell)$.
\end{itemize}
Therefore  $\mathcal G'$ satisfies (\ref{wingdeg'}) and the proof is complete.
\end{proof}

\section{Acknowledgement}
This research was carried out while the author was a PhD student at Auburn University. The author is deeply grateful to his supervisor Professor Chris Rodger, his colleague Joe Chaffee, and the anonymous referee for their constructive comments.



\bibliographystyle{model1-num-names}
\bibliography{<your-bib-database>}







\end{document}